\documentclass[12pt]{amsart}
\usepackage{amscd,amssymb,epsfig,amsthm}
\usepackage[labelfont=bf,labelsep=space]{caption}

\calclayout
\newif\ifpdf
\ifx\pdfoutput\undefined
  \pdffalse 
\else 
  \pdftrue\pdfoutput=1 
\fi 
 
\numberwithin{equation}{section}

\newcommand{\R}{\mathbb{R}}

\renewcommand\P{{\mathcal P}}

\newcommand\Q{{\mathcal Q}}
\renewcommand\S{{\mathcal S}}

\newcommand\x{\times}

\newtheorem{thm}{Theorem}[section]

\newtheorem{defn}[thm]{Definition} 
\newtheorem*{defnx}{Definition 2.1$'$}

\def\whsq{\vbox to 5.8pt 
{\offinterlineskip\hrule 
\hbox to 5.8pt{\vrule height 
5.1pt\hss\vrule height 5.1pt}\hrule}}

\def\<{\langle} 
\def\>{\rangle} 
\def\PP{{\mathop{{\rm I}\kern-.2em{\rm P}}\nolimits}} 
\def\FF{{\mathop{{\rm I}\kern-.2em{\rm F}}\nolimits}}   
\def\ZZ{{\mathop{{\rm I}\kern-.2em{\rm Z}}\nolimits}}

\DeclareMathOperator{\spn}{\operatorname{span}}

\newcommand\range{\mathbb N_n}

\newlength{\sidemargin} 
\begin{document}
\title[]{
The serendipity family of finite elements
}
\author{Douglas N. Arnold and Gerard Awanou}

\address{Department of Mathematics, University of Minnesota, Minneapolis,
Minnesota 55455}

\email{arnold@umn.edu}
\urladdr{http://www.ima.umn.edu/\~{}arnold}

\address{Northern Illinois University,
Department of Mathematical Sciences,
Dekalb, IL, 60115}

\email{awanou@math.niu.edu}  

\urladdr{http://www.math.niu.edu/\~{}awanou}

\thanks{The work of the first author was supported in part by NSF grant
DMS-0713568 and the work of the second author was supported in part by
NSF grant DMS-0811052 and the Sloan Foundation.}
\thanks{
To appear in Foundations of Computational Mathematics, 2011.}

\begin{abstract}
We give a new, simple, dimension-independent definition of the serendipity finite element family.
The shape functions are the
span of all monomials which are linear in at least $s-r$ of the
variables where $s$ is the degree of the monomial or, equivalently, whose superlinear
degree (total degree with respect to variables entering at least quadratically) is at most $r$.  The degrees of freedom are given
by moments of degree at most $r-2d$ on each face of dimension $d$.  We establish unisolvence
and a geometric decomposition of the space.
\end{abstract}

\maketitle
\section{Introduction}
The serendipity family of finite element spaces are among the most popular finite element spaces
for parallelepiped meshes in two, and, to a lesser extent,
three dimensions.  For each such mesh and
each degree $r\ge 1$ they provide a finite element subspace with $C^0$ continuity which has
significantly smaller dimension than the more obvious alternative, the tensor product
Lagrange element family.
However, the serendipity elements
are rarely studied systematically, particularly in 3-D.  Usually only
the lowest degree examples are discussed, with the pattern for higher degrees not evident.
In this paper, we give a simple, but apparently new, definition of the serendipity
elements, by specifying in a dimension-independent fashion the space of shape functions
and a unisolvent set of degrees of freedom.

The serendipity finite element space $\S_r$ may be viewed as a reduction of
the space $\Q_r$, the
tensor product Lagrange finite element space of degree $r\ge1$.
The $\Q_r$ elements are certainly the simplest, and for many purposes the best, 
$C^0$ finite elements on parallelepipeds.
They may be defined by specifying a space of polynomial shape functions,
$\Q_r(I^n)$, and a unisolvent set of degrees of freedom, for the unit cube $I^n$.
Here $I=[-1,1]$ is the
unit interval, $n\ge1$ is the space dimension, and
for $f\subset \R^n$ and $r\ge0$ the space $\Q_r(f)$ is defined as the restriction to $f$ of
the functions on $\R^n$ which are polynomial of degree at most $r$ in
each of the $n$ variables separately, so $\dim \Q_r(I^n)=(r+1)^n$.
In addition to the usual evaluation degrees of freedom associated  to each vertex, to each
face $f$ of $I^n$ of some dimension $d\ge 1$ are associated the degrees of freedom
\begin{equation}\label{dofq}
u \mapsto \int_f uq,\quad q\in\Q_{r-2}(f),
\end{equation}
(properly speaking, $q$ ranges over a basis of $\Q_{r-2}(f)$, but we
shall henceforth dispense with this distinction).
Since the number of faces of dimension
$d$ of $I^n$ is $2^{n-d}\binom{n}{d}$, and, by the binomial expansion,
$$
\sum_{d=0}^n 2^{n-d}\binom{n}{d}(r-1)^d = (r+1)^n,
$$
we see that the total number of degrees of freedom coincides with the dimension of $\Q_r(I^n)$.
The proof of unisolvence is straightforward, by using the degrees of freedom to show, in turn,
that $u$ vanishes on the faces of dimension $0,1,\ldots$.


The serendipity family in two dimensions is discussed in most finite element textbooks, especially
the lowest order cases, namely
the 4-node, 8-node, and 12-node rectangular elements.  The idea is to maintain the
same degrees of freedom as for $\Q_r$ on the boundary of the element, but to remove
interior degrees of freedom, and specify a correspondingly smaller shape function space
for which the smaller set of degrees of freedom is unisolvent.  In this way,
one obtains a space of lower dimension without sacrificing $C^0$ continuity and,
hopefully, without much loss of accuracy.
For $r=1$,
there are no interior degrees of freedom for $\Q_r$, so the space $\S_1$ is identical
to $\Q_1$. But the serendipity space $\S_2$ and $\S_3$ have only $8$ and $12$
degrees of freedom, respectively, compared to $9$ and $16$, respectively, for
$\Q_2$ and $\Q_3$.  A possible generalization to higher degree is to keep
only the $4r$ boundary degrees of freedom of $\Q_r$ and to seek $\S_r$ as a
subspace of $\Q_r$ of dimension $4r$ for which these degrees of freedom are
unisolvent.  This is easily accomplished by taking $\S_r$ to be the span
of the monomials $x^i$, $x^iy$, $y^i$, and $xy^i$, for $0\le i\le r+1$, ($4r$
monomials altogether, after accounting for duplicates).  The resulting finite element
space is referred to as the serendipity space in some of the literature,
e.g., \cite{MR1008473}.  However,
for $r>3$ this space does not contain the complete polynomial space $\P_r$, and
so does not achieve the same degree of approximation as $\Q_r$.  Therefore
the shape functions for the serendipity space in two dimensions is usually taken
to be the span of $\P_r(I^n)$ together with the above monomials, or equivalently,
$$
\S_r(I^n) = \P_r(I^n)+\spn[x^ry,xy^r].
$$
The degrees of freedom associated to the vertices and other faces of positive codimension
are taken to
be the same as for $\Q_r$, and the degrees of freedom in the interior of the element
can be taken as the moments $u\mapsto\int_{I^2} uq$, $q\in\P_{r-4}(I^2)$, resulting
in a unisolvent set.

This definition does not generalize in an obvious fashion to three (or more)
dimensions, but many texts discuss the
lowest order cases of serendipity elements in three dimensions: the 20-node brick, and possibly
the 32-node brick \cite{MR1008473, Kaliakin, MR0443377,MR1897985},
which have the same degrees of freedom as $\Q_r(I^3)$ on the boundary, $r=2,3$, but none in
the interior.  It is often remarked that the choice of shape function space is not
obvious, thus motivating the name ``serendipity.''
The pattern to extend these low degree cases to higher degree brick elements is not
evident and usually not discussed.  A notable exception is the text of
Szab\'o and Babu\v{s}ka \cite{MR1164869}, which defines the space of
serendipity polynomials on the three-dimensional
cube for all polynomial degrees, although without using the term serendipity and by an
approach quite different from that given here.  High-degree serendipity elements
on bricks have been used in the $p$-version of the finite element method \cite{MR1067945}.

In this paper we give a simple self-contained dimension-independent
definition of the serendipity family.  For general $n,r\ge 1$,
we define the polynomial space $\S_r(I^n)$ and the degrees of freedom associated to each
face of the $n$-cube and prove unisolvence.

\section{Shape functions and degrees of freedom}
\subsection{Shape functions}
We now give, for general dimension $n\ge1$ and general degree $r\ge 1$, a concise definition
of the space $\S_r(I^n)$ of shape functions for the serendipity finite element. 
(As usual,
a monomial is said to be linear in some variable $x_i$ if it is divisible by $x_i$
but not $x_i^2$, and it is said to be superlinear if it is divisible by $x_i^2$.)
\begin{defn}
The serendipity space $\S_r(I^n)$ is the
span of all monomials in $n$ variables which are linear in at least $s-r$ of the
variables where $s$ is the degree of the monomial.
\end{defn}
\noindent
We may express this definition in an alternative form using the notion of the
\emph{superlinear degree}.
Define the superlinear degree of a monomial $p$, denoted $\deg_2 p$, to
be the total degree of $p$ with respect to variables which enter it superlinearly
(so, for example, $\deg_2 x^2 y z^3 = 5$) and define the superlinear
degree of a general polynomial
as the maximum of the superlinear degree of its monomials.
\begin{defnx}
The serendipity space $\S_r(I^n)$ is the
space of all polynomials in $n$ variables with superlinear degree at most $r$.
\end{defnx}
\noindent
It is easy to see that Definitions~2.1 and 2.1$'$ are equivalent, since
for a monomial $p$ which is linear in $l$ variables, $\deg p = \deg_2 p + l$.
Surprisingly, neither form of the definition seems to appear in the literature.

Since any monomial is, trivially, linear in at least $0$ variables, and no monomial in
$n$ variables is linear in more than $n$ variables, we have immediately from Definition~2.1
that $\P_r(I^n)\subset\S_r(I^n)\subset\P_{r+n}(I^n)$, where $\P_r(I^n)$ is the
space of polynomial functions of degree at most $r$ on $I^n$.  In fact, $\S_r(I^n)\subset\P_{r+n-1}(I^n)$, since the only monomial which is linear in all $n$ variables
is $x_1x_2\cdots x_n$, which is of degree $n$, not degree $r+n$.
In particular, the one-dimensional case is trivial: $\S_r(I)=\P_r(I)=\Q_r(I)$, for all $r\ge 1$.
The two-dimensional case is simple as well:
$\S_r(I^2)$ is spanned by $\P_r(I^2)$ and the two monomials $x_1x_2^r$ and $x_1^rx_2$ (these
two coincide when $r=1$).  Thus we recover the usual serendipity shape functions
in two dimensions.
In three dimensions, $\S_r(I^3)$ is obtained by adding to $\P_r(I^3)$ the span
of certain monomials of degrees $r+1$ and $r+2$, namely those of degree $r+1$ which are
linear in at least one of the $n$ variables, and those of degree $r+2$ which are linear
in at least two of them (there are three of these---$x_1x_2x_3^{r}$, $x_1x_2^{r}x_3$,
$x_1^{r}x_2x_3$---except for $r=1$, when all three coincide).

To calculate $\dim\S_r(I^n)$ we count the monomials in $n$ variables
with superlinear degree
at most $r$.  For any monomial $p$ in $x_1,\ldots,x_n$, 
let $J\subset\range:=\{1,\ldots,n\}$ be the set of indices for which $x_i$ enters $p$ superlinearly
and let $d\ge0$ be the
cardinality of $J$.  Then 
$$
p= \bigl(\prod_{j\in J} x_j^2\bigr)\x q \x \bigl(\prod_{i\in J^c} x_i^{a_i}\bigr),
$$
where $q$ is a monomial in the $d$ variables indexed by $J$ and each
$a_i$, $i\in J^c$ (the complement of $J$), equals either $0$ or $1$.  Note that $\deg_2p = 2d+\deg q$.
Thus we may uniquely specify a monomial in $n$ variables with superlinear degree
at most $r$
by choosing $d\ge 0$, choosing a set $J$ consisting of $d$ of the $n$ variables
(for which there are $\binom{n}{d}$ possibilities), choosing a monomial
of degree at most $r-2d$ in the $d$ variables ($\binom{r-d}{d}$
possibilities), and choosing the exponent to be
either $0$ or $1$ for the $n-d$ remaining indices
($2^{n-d}$ possibilities).  Thus
\begin{equation}\label{dimen}
\dim\S_r(I^n) = \sum_{d=0}^{\min(n,\lfloor r/2\rfloor)} 2^{n-d}\binom{n}{d}
\binom{r-d}{d}.
\end{equation}
Table~1 shows the dimension for small values of $n$ and $r$.
\begin{table}[t]
 \begin{tabular}{c|cccccccc}
 \multicolumn{4}{c}{} & \multicolumn{2}{c}{$r$}\\
 $n$ & 1 & 2 & 3 & 4 & 5 & 6 & 7 & 8\\
\hline
1 & 2 & 3& 4 & 5 & 6 & 7 & 8 & 9\\
2 & 4 & 8 & 12 & 17 & 23 & 30 & 38 & 47\\
3 & 8 & 20 & 32 & 50 & 74 & 105 & 144 & 192\\
4 & 16 & 48 & 80 & 136 & 216 & 328 & 480 & 681\\
5 & 32 & 112 & 192 & 352 & 592 & 952 & 1472 & 2202
 \end{tabular}
\bigskip
\caption{$\dim S_r(I^n)$}
\end{table}

\subsection{Degrees of freedom}
We complete the definition of the serendipity finite elements by specifying
a set of degrees of freedom and proving that they are unisolvent.  Let $f$ be
a face of $I^n$ of dimension $d\ge 0$.  Then the degrees of freedom associated
to $f$ are given by
\begin{equation}\label{serdof}
u \mapsto \int_f uq, \quad q\in \P_{r-2d}(f).
\end{equation}
Note that $\P_s(f)$ is defined to be the space of restrictions
to $f$ of $\P_s(I^n)$, so if $f$ is a vertex, then $\P_s(f)=\R$ for all $s\ge 0$.
In this case the integral is with respect to the counting measure,
so each vertex is assigned the evaluation degree of freedom.

In contrast to the degrees of freedom \eqref{dofq} of the tensor product
Lagrange family, for the serendipity family, the degrees of freedom on faces are
given in terms of the $\P_s(f)$ rather than $\Q_s(f)$.  For example, for $\S_6(I^2)$, there
are $6=\dim\P_{6-2\times 2}(I^2)$ degrees of freedom internal to
the square.  For $\S_6(I^3)$,
there are $6$ degrees of freedom on each face of the cube and one degree of freedom
internal to the cube.
Figure~\ref{fg:seren} shows degree of freedom diagrams for the spaces
$\S_r(I^n)$ for $r\le4$ and $n\le3$.
\begin{figure}[tb]
 \includegraphics[width=6in]{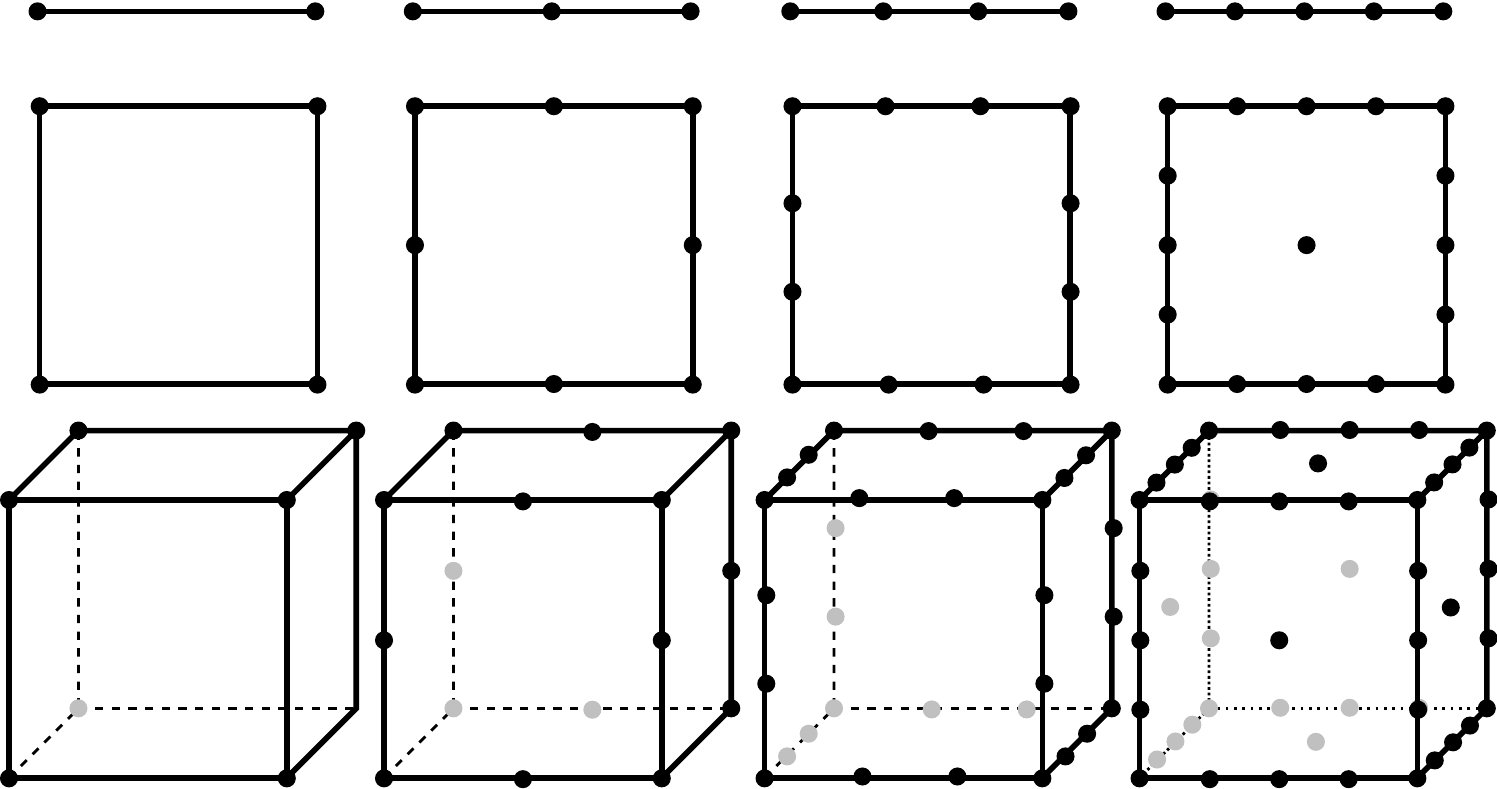}
\caption[]{Degrees of freedom for $\S_r(I^n)$ for $r=1,2,3,4$, $n=1,2,3$.
Dots indicate the number of degrees of freedom associated to each
face.}\label{fg:seren}
\end{figure}

\begin{thm}[Unisolvence] The degrees of freedom specified in \eqref{serdof} are unisolvent
 for $\S_r(I^n)$.
\end{thm}
\begin{proof}
 First we note that number of degrees of freedom equals the dimension
of the space.
Indeed, there are $2^{n-d}\binom{n}{d}$ faces of the cube of dimension $d$ (obtained
by fixing $n-d$ of the variables to $\pm 1$) and, for $f$ of dimension
$d$, $\dim\P_{r-2d}(f)=\binom{r-d}{d}$, so the total number of degrees of
freedom proposed is precisely $\dim\S_r(I^n)$ given by \eqref{dimen}. 
It remains to show that if $u\in\S_r(I^n)$ and all the
quantities in \eqref{serdof} vanish, then $u$ vanishes.  We do this by induction
on $n$, the case $n=1$ being trivial.  Let $F$ be a face of $I^n$ of dimension $n-1$
and let $U=u|_F$.  The $U$ is a polynomial in $n-1$ variables with superlinear degree
at most $r$, i.e., $U\in\S_r(F)$.  Moreover, if $f$ is any face of $F$ of dimension
$d$, and $q\in\P_{r-2d}(f)$, then $\int_f Uq = \int_f uq=0$.  Therefore, by the
inductive hypothesis, $U$ vanishes.  In this way, we see that $u$ vanishes on all its
faces, i.e., whenever we fix some $x_i$ to $\pm 1$.  Therefore
$$
u=(1-x_1^2)\cdots(1-x_n^2)\, p
$$
for some polynomial $p$.  Note that $\deg_2 u = 2n + \deg p$.  Therefore $\deg p \le r-2n$,
and we make take $f=I^n$ and $q=p$ in \eqref{serdof} to find
$$
\int_{I^n}(1-x_1^2)\cdots(1-x_n^2) \, p^2 = 0.
$$
It follows that $p=0$.
\end{proof}

In the proof of unisolvence we established that the degrees of freedom associated to
a face and its subfaces determine the restriction of $u$ to the face.  This is important,
since it implies that
an assembled serendipity finite element function is continuous.

\section{Geometric decomposition}
In this section we give a geometric decomposition of $\S_r(I^n)$,
by which we mean a direct sum decomposition into subspaces associated to
the faces.  Such a decomposition can be used to derive explicit
local bases which are useful for the efficient implementation
of the elements, and also for insight.

First we introduce some notation.
Let $\Delta_d(I^n)$ denote the set of faces of the $n$-cube of dimension $d$
and $\Delta(I^n)=\bigcup_{d=0}^n\Delta_d(I^n)$ the set of all faces of all dimensions.
A face $f\in\Delta_d(I^n)$ is determined by the equations $x_j=c_j$ for $j\in J$
where $J \subset \range$ is a set of cardinality $n-d$ and each $c_j\in\{-1,1\}$.
We define the bubble function for $f$ as
$$
b_f = \prod_{j \in J^c} (1-x_j^2)\prod_{j \in J} (1+c_jx_j),
$$
the unique (up to constant multiple) nontrivial
polynomial of lowest degree vanishing on the $(n-1)$-dimensional faces of $I^n$ which
do not contain $f$.  Note that $b_f$ is strictly positive on
the relative interior of $f$ and vanishes on all faces of $I^n$ which do not contain $f$.
For $s\ge0$ we denote by $\P_s^f(I^n)$ the space of
polynomials of degree at most $s$ in the $d$ variables $x_j$, $j\in J^c$.  If $d=0$, then
$\P_s^f(I^n)$ is understood to be $\R$.

To a face $f$ of dimension $d$ we associate the space of polynomials
$V_f := \P_{r-2d}^f(I^n)b_f$. By definition,
any element $q\in V_f$ has the form $q=p\prod_{j \in J} (x_j+c_j)$ where $p\in\P_r^f(I^n)$,
and so $\deg_2 q \le r$.  Thus $V_f\subset\S_r(I^n)$.
The following theorem states that they do indeed form a geometric decomposition.

\begin{thm}[Geometric decomposition of the serendipity space]
Let $\S_r(I^n)$ denote the serendipity space of degree $r$, and for each $f\in\Delta_d(I^n)$,
let $V_f= \P_{r-2d}^f(I^n)b_f$.  Then
\begin{equation} \label{geo1}
\S_r(I^n) = \sum_{f\in\Delta(I^n)}V_f.
\end{equation}
Moreover, the sum is direct.
\end{thm}
\begin{proof}
Clearly $\dim V_f=\dim P_{r-2d}^f(I^n)=\binom{r-d}{d}$, and so \eqref{dimen} implies that
$$
\dim\S_r(I^n) = \sum_{f\in\Delta(I^n)}\dim V_f.
$$
Hence, it is sufficient to prove that if $p$ is a monomial with superlinear degree $\le r$,
then $p\in\sum_f V_f$.
Write $p=x_1^{\alpha_1}\cdots x_n^{\alpha_n}$ and let
\begin{equation*}
 J_0=\{\,i\in\range\,|\, \alpha_i=0\,\},\quad
 J_1=\{\,i\in\range\,|\, \alpha_i=1\,\},\quad
 J_2 = (J_0\cup J_1)^c,
\end{equation*}
denote the sets indexing the variables in which $p$ is constant, linear, and superlinear, respectively.
By assumption, $\sum_{j\in J_2}\alpha_j \le r$.
Now we expand
\begin{equation*}
x_j^{\alpha_j} = 
 \begin{cases}
 \frac12(1+x_j)+\frac12(1-x_j), & j\in J_0, \\
\frac12(1+x_j)-\frac12(1-x_j), & j\in J_1, \\
\frac12(1+x_j) + \frac{(-1)^{\alpha_j}}2(1-x_j) + (1-x_j^2)r_j(x_j), & j\in J_2,  
 \end{cases}
\end{equation*}
where $\deg r_j = \alpha_j-2$, and insert these into $p=x_1^{\alpha_1}\cdots x_n^{\alpha_n}$.
We find that $p$
is a linear combination of terms of the form
\begin{equation}\label{defq}
 q = \prod_{j\in J'} r_j(x_j)(1-x_j^2) \prod_{j\in J''\cup J_0 \cup J_1}(1\pm x_j),
\end{equation}
where $J'$ and $J''$ partition $J_2$.  Let $f$ be the face given by $1\mp x_j=0$, $j\in J''\cup J_0\cup J_1$,
where the signs are chosen opposite to those on \eqref{defq}.  Then $\dim f= \#J'$ and
$q = rb_f$ with $r\in \P_s^f(I^n)$, $s=\sum_{j\in J'}(\alpha_j-2)\le r-2\dim f$.  This shows
that $q\in V_f$ and $p \in \sum_f V_f$, as desired.
\end{proof}

\bibliographystyle{amsplain}
\bibliography{serendipity}

\providecommand{\bysame}{\leavevmode\hbox to3em{\hrulefill}\thinspace}
\providecommand{\MR}{\relax\ifhmode\unskip\space\fi MR }
\providecommand{\MRhref}[2]{%
  \href{http://www.ams.org/mathscinet-getitem?mr=#1}{#2}
}
\providecommand{\href}[2]{#2}
\begin{thebibliography}{1}

\bibitem{MR1008473}
Thomas J.~R. Hughes, \emph{The finite element method}, Prentice Hall Inc.,
  Englewood Cliffs, NJ, 1987, Linear static and dynamic finite element
  analysis, With the collaboration of Robert M. Ferencz and Arthur M. Raefsky.

\bibitem{Kaliakin}
Victor~N. Kaliakin, \emph{Introduction to approximate solution techniques,
  numerical modeling, \& finite element methods}, CRC, 2001, Civil and
  Environmental Engineering.

\bibitem{MR1067945}
Jan Mandel, \emph{Iterative solvers by substructuring for the {$p$}-version
  finite element method}, Comput. Methods Appl. Mech. Engrg. \textbf{80}
  (1990), no.~1-3, 117--128, Spectral and high order methods for partial
  differential equations (Como, 1989).

\bibitem{MR0443377}
Gilbert Strang and George~J. Fix, \emph{An analysis of the finite element
  method}, Prentice-Hall Inc., Englewood Cliffs, N. J., 1973, Prentice-Hall
  Series in Automatic Computation.

\bibitem{MR1164869}
Barna Szab{\'o} and Ivo Babu{\v{s}}ka, \emph{Finite element analysis}, A
  Wiley-Interscience Publication, John Wiley \& Sons Inc., New York, 1991.

\bibitem{MR1897985}
O.~C. Zinkiewicz, R.~L. Taylor, and J.~Z. Zhu, \emph{The finite element method:
  its basis and fundamentals, vol. 1}, sixth ed., Butterworth-Heinemann, 2005.

\end{thebibliography}

\end{document}